\documentclass{article}
\usepackage[utf8]{inputenc}
\usepackage{graphicx} 

\title{Random Tur\'an theorem for hypergraph cycles}

\author
{Dhruv Mubayi\thanks{Department of Mathematics, Statistics, and Computer Science, University of Illinois, Chicago, IL, 60607 USA. Email: mubayi@uic.edu.
Research partially supported by NSF awards DMS-1763317 and DMS-1952767.}
\and Liana Yepremyan \thanks{Department of Mathematics, Statistics, and Computer Science, University of Illinois at Chicago, Chicago, USA,
London School of Economics, Department of Mathematics, London, UK,
e-mail:lyepre2@uic.edu, l.yepremyan@lse.ac.uk, 
Research supported by Marie Sklodowska Curie Global Fellowship, H2020-MSCA-IF-2018:846304}}
\date{\today}

\usepackage{hyperref}
\usepackage{float}
\usepackage{amsmath, color, amssymb, amsfonts, amsthm, enumerate, a4wide}

\usepackage{cleveref}

\newcommand{\eps}{\ensuremath{\varepsilon}}

\newtheorem{theorem}{Theorem}[section] 

\newtheorem{claim}[theorem]{\bf Claim}
\newtheorem{lemma}[theorem]{\bf Lemma}
\newtheorem{conjecture}[theorem]{\bf Conjecture}

\newtheorem{proposition}[theorem]{\bf Proposition}
\newtheorem{corollary}[theorem]{\bf Corollary}

\newcommand\numberthis{\addtocounter{equation}{1}\tag{\theequation}}
\newcommand\numeq[1]%
  {\stackrel{\scriptscriptstyle(\mkern-1.5mu#1\mkern-1.5mu)}{=}}
  
\newcommand{\Prob}{\mathbb{P}}
\newcommand{\E}{\mathbb{E}}
\newcommand{\ex}{{\rm ex}}

\begin{document}
\setlength{\baselineskip}{12pt}

\maketitle
\begin{abstract}  
Given $r$-uniform hypergraphs $G$ and $H$ the Tur\'an number $\ex(G, H)$   is the  maximum number of edges in an  $H$-free subgraph of $G$.   We study the typical value of  $\ex(G, H)$ when $G=G_{n,p}^{(r)}$, the Erd\H{o}s-R\'enyi random $r$-uniform hypergraph,   and $H=C_{2\ell}^{(r)}$, the $r$-uniform linear cycle of length $2\ell$. The case of graphs ($r=2$) is a longstanding open problem that has been investigated by many researchers.
We determine $\ex(G_{n,p}^{(r)}, C_{2\ell}^{(r)})$ up to polylogarithmic factors for all but a small interval of values of $p=p(n)$ whose length decreases as $\ell$ grows. 

Our main technical contribution is a balanced supersaturation result for linear even cycles which  improves upon previous such results by Ferber-Mckinley-Samotij  and Balogh-Narayanan-Skokan. The novelty  is that   the supersaturation result depends on the codegree of some pairs of vertices   in the underlying hypergraph. This approach could be used to prove similar results for other hypergraphs $H$.

\end{abstract}
\begin{section}{Introduction}
Write $e(G)$ for the number of edges in hypergraph $G$. Let $G$ and $H$ be $r$-uniform hypergraphs (henceforth $r$-graphs). The Tur\'an number $\ex(G, H)$  is the   maximum of $e(G')$ over all $H$-free subgraphs of $G' \subset G$. When $G=K_{n}^{(r)}$, the complete $r$-graph on $n$ vertices, $\ex(G, H)$ is simply denoted by $\ex(n,H)$. Determining $\ex(n,H)$  and its order of magnitude for large $n$ is a central problem in extremal (hyper)graph theory, known as the Tur\'an problem of $H$. For more on Tur\'an numbers, we refer the reader to the excellent surveys~\cite{FS} for graphs and~\cite{keevash} for hypergraphs.

In this paper we study $\ex(G,H)$ when $G$ is the random $r$-graph $G_{n,p}^{(r)}$, and $H$ is a linear even cycle. Here $G_{n,p}^{(r)}$ is the $r$-graph  on $n$ labelled vertices whose edges are independently present with probability $p=p(n)$. We use standard asymptotic notation.  Given functions $f,g: \mathbb R^+ \to \mathbb R^+$,
  we write $f(n)\ll g(n)$ to mean $f(n)/g(n)\rightarrow 0$ as $n\rightarrow \infty$, $f(n) = O(g(n))$ to mean that there is an absolute positive constant $C$ such that $f(n) < C g(n)$, $f(n)=\Omega(n)$ to mean that $g(n) = O(f(n))$ and $f(n) = \Theta(g(n))$ to mean that $f(n)=O(g(n))$ and $f(n)= \Omega(n)$. 
 Throughout this paper, we say that a statement depending on $n$ holds asymptotically almost surely (abbreviated a.a.s.) if the probability that it holds tends to $1$ as $n$ tends to infinity.
 All our theorems will be statements that hold a.a.s. in the probability space $G_{n,p}^{(r)}$. 

 The random variable $\ex(G_{n,p},H) = \ex(G_{n,p}^{(2)}, H)$ was first considered by
Babai, Simonovits, and Spencer~\cite{BSS} who treated the case in which $H$ has chromatic
number three and $p$ is a constant. The systematic study of $\ex(G_{n,p},H)$  was initiated by Kohayakawa, Luczak and R\"{o}dl~\cite{KLR} (see the survey~\cite{RS} for more extremal results in random graphs). One of their  conjectures resolved independently by Conlon
and Gowers~\cite{CG} and by Schacht~\cite{S} determines the asymptotic value of $\ex(G_{n,p}, H)$ whenever $H$
has chromatic number at least three.

The behaviour of $\ex(G_{n,p}, H)$ when $H$ is bipartite is a wide open problem that is closely  related to  the order of magnitude of the usual Tur\'an numbers $\ex(n, H)$. One case of bipartite $H$ that has been extensively studied is when $H=C_{2\ell}$, the even cycle on $2\ell$ vertices. Haxell, Kohayakawa and \L{}uczak~\cite{HKL} determined the so-called \emph{threshold}  $p$ for $H=C_{2\ell}$.  Namely they showed that a.a.s. if $p\gg n^{-1+1/(2\ell-1)} $ then $\ex(G_{n,p}, C_{2\ell})\ll e(G_{n,p})$, and if $p\ll n^{-1+1/(2\ell-1)}$  then  $\ex(G_{n,p},C_{2\ell})=(1-o(1))e(G_{n,p})$. Kohayakawa, Kreuter and Steger~\cite{KKS} improved on the second part and obtained more precise bounds for a certain range of $p$.
Finally, using the container method, Morris and Saxton~\cite{MS} further improved the upper bounds on  $\ex(G_{n,p}, C_{2\ell})$ for a broader range of $p$.

\begin{theorem} [\cite{HKL, KKS, MS}] 
\label{thm:graphcycles}For every $\ell\geq 2$, there exists  $C=C(\ell)$ such that a.a.s.
$$ \ex\left(G_{n,p}, C_{2\ell}\right) \leq \begin{cases} n^{1+1/(2\ell-1)} (\log{n})^{2},  &   \mbox{if } p \leq n^{-(\ell-1)/(2\ell-1)}(\log{n})^{2\ell} ,\\
Cp^{1/\ell}n^{1+1/\ell},& \mbox{otherwise.} \end{cases}$$

\end{theorem}
In Theorem~\ref{thm:graphcycles}, the first bound is sharp up to a polylog factor by the results of \cite{KKS}. As for the second bound, an old conjecture of Erd\H{o}s and Simonovits says that  there is a graph of girth at least $2\ell+1$ and  $\Omega(n^{1+1/\ell})$ edges. If this conjecture is true, then the second bound is also  sharp up to the value of the constant $C$ (see the discussion after Conjecture 2.3 in~\cite{MS}).

The problem of finding the largest $H$-free subgraph of $G_{n,p}^{(r)}$ is closely related to the problem of determining $|Forb(n,H)|$, the number of $H$-free subgraphs on $n$ labelled vertices.   There is a well-developed theory for the latter problem for  $r$-graphs that are not $r$-partite~\cite{EFR, NRS, NR}. 
The corresponding question for $r$-partite $r$-graphs was initiated in a recent paper of the first author and Wang~\cite{MW}. The $r$-graph $C_k^{(r)}$ is obtained from the graph $k$-cycle $C_k$  by adding $r-2$ new vertices of degree one to each graph edge (thus enlarging each graph edge to an $r$-graph edge).
The authors in~\cite{MW} determined the asymptotics  of $|Forb(n, C_{k}^{(r)})|$ for even $k$ and $r=3$, and conjectured that similar results hold for all $k,r\ge 3$. This was later confirmed by Balogh, Narayanan and Skokan~\cite{BNS}.   Soon after, Ferber, McKinley and Samotij~\cite{FMS} proved similar results for a  much larger class of $r$-graphs that includes linear cycles and linear paths. However, the results of~\cite{BNS, FMS} both rely on supersaturation theorems that are not strong enough to imply  anything nontrivial for $\ex(G_{n,p}^{(r)}, C_k^{(r)})$ when $k>3$ and $p=o(1)$. 

In this paper, we prove a stronger supersaturation result that can be used  to compute $\ex(G_{n,p}^{(r)}, C_k^{(r)})$ for a large range of $p$ when $k$ is even. The bounds in our supersaturation result  depend on  the codegree of some pairs of vertices in the underlying hypergraph. We expect that this approach can be applied to  $r$-partite $r$-graphs other than even cycles.

\section{Main result}

 Our main result is the following extension of 
Theorem~\ref{thm:graphcycles} to linear even cycles $C_{2\ell}^{(r)}$.

\begin{theorem}\label{thm:ourresult}For every $\ell\geq 2$ and $r\geq 3$ a.a.s. the following holds: 
$$ \ex\left(G_{n,p}^{(r)}, C_{2\ell}^{(r)}\right) \leq \begin{cases}  p^{\frac{1}{(2\ell-1)}}n^{1+\frac{r-1}{2\ell-1}+o(1)}, & \mbox{if } {n^{-(r-2)+o(1)}\leq p\leq n^{-(r-2)+\frac{1}{2\ell-2}+o(1)}} \\ 
pn^{r-1+o(1)},& \mbox{otherwise.} \end{cases}$$
All $o(1)$ error terms in the exponents are  $O(\log\log n /\log n)$.
\end{theorem}

The upper bounds in Theorem~\ref{thm:ourresult} together with monotonicity quickly give us the order of magnitude of $\ex(G_{n,p}^{(r)}, C_{2\ell}^{(r)})$ apart from polylog factors for all but a small range of $p$.


\begin{corollary}\label{thm:ourcor}For every $\ell\geq 2$ and $r\geq 3$ a.a.s. the following holds: 
$$ \ex\left(G_{n,p}^{(r)}, C_{2\ell}^{(r)}\right) = \begin{cases} 
\Theta(pn^r),  & \mbox{if } {n^{-r} \ll p\ll n^{-(r-1)+\frac{1}{2\ell-1}}} \\ 
n^{1+\frac{1}{2\ell-1}+o(1)}, & \mbox{if } {n^{-(r-1)+\frac{1}{2\ell-1}+o(1)}\leq p\leq n^{-(r-2)+o(1)}} \\ 
pn^{r-1+o(1)}, & \mbox{if } { p \geq n^{-(r-2)+\frac{1}{2\ell-2}+o(1)}}. \end{cases}$$
For $n^{-(r-2)+o(1)} \leq p \leq n^{-(r-2)+\frac{1}{2\ell -2}}$, 
$$\max\{n^{1+\frac{1}{2\ell-1}+o(1)}, \, pn^{r-1+o(1)}\} \leq \ex(G_{n,p}^{(r)}, C_{2\ell}^{(r)}) \leq p^{\frac{1}{(2\ell-1)}}n^{1+\frac{r-1}{2\ell-1}+o(1)}.$$
\end{corollary}

\begin{proof} When $n^{-r}\ll p \ll {n^{-(r-1)+\frac{1}{2\ell-1}}}$ a.a.s. $G_{n,p}^{(r)}$ has a $C_{2\ell}^{(r)}$-free subgraph with $(1+o(1))e(G_{n,p}^{(r)})$  edges by a simple deletion argument (see Proposition~\ref{prop:lowerbound1} for details), and this is  best possible, therefore in this regime  $\ex(G_{n,p}^{(r)}, C_{2\ell}^{(r)}) =\Theta(pn^r)$. For $p\gg n^{-(r-1)}$, $G_{n,p}^{(r)}$ a.a.s contains $\Omega(pn^{r-1})$ edges  containing a fixed vertex (see Proposition~\ref{prop:lowerbound2} for the proof).  This together with the second bound in Theorem~\ref{thm:ourresult} determines the optimal behaviour of  $\ex(G_{n,p}^{(r)}, C_{2\ell}^{(r)})$ up to $n^{o(1)}$ factors for $p\geq n^{-(r-2)+\frac{1}{2\ell-2}+o(1)}$, that is, in this range  $\ex(G_{n,p}^{(r)}, C_{2\ell}^{(r)})=  pn^{r-1+o(1)}$. To understand the behaviour of $\ex(G_{n,p}^{(r)}, C_{2\ell}^{(r)})$  when $n^{-(r-1)+\frac{1}{2\ell-1}+o(1)}\leq p \leq n^{-(r-2)+\frac{1}{2\ell-2}+o(1)}$   we need to combine our first upper bound in Theorem~\ref{thm:ourresult} together with monotonicity of the  $H$-freeness property. 

\begin{figure}  \centering
    \includegraphics[width=1.1\textwidth]{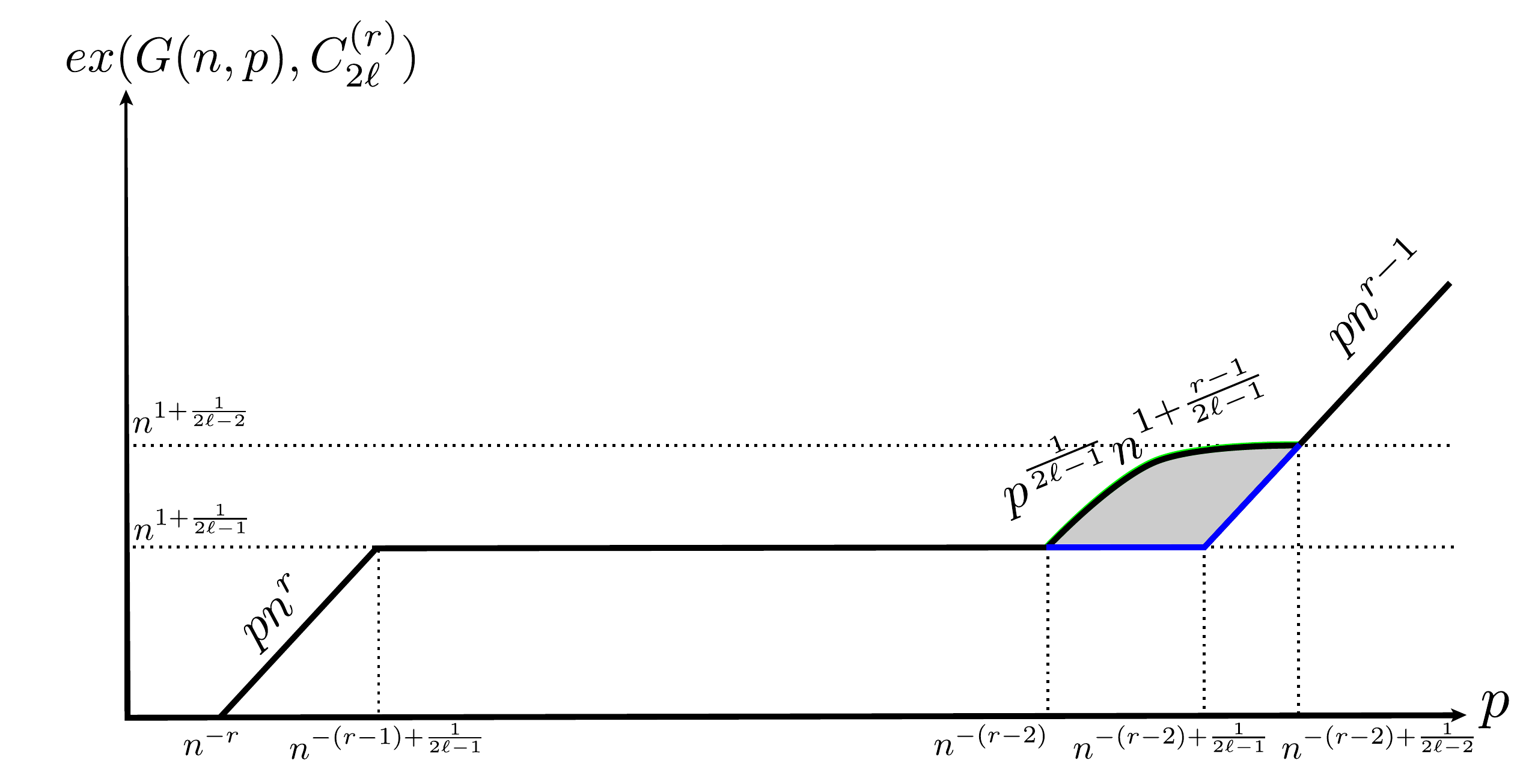}
  \caption{The behaviour of $\ex(G_{n,p}^{(r)}, C_{2\ell}^{(r)})$}
  \label{fig:general}
\end{figure}

Indeed, the property ``$G_{n,p}^{(r)}$ has an $H$-free subgraph with $m$ edges'' is monotone with respect to $p$,  that is if $q>p$, then if a.a.s. $G_{n,p}^{(r)}$ has $H$-free subgraph with $m$ edges then so does $G_{n,q}^{(r)}$. Therefore, if we let $p_0=n^{-(r-1)+\frac{1}{2\ell-1}+o(1)}$ then by  our earlier discussion we know that 
$$\ex(G_{n,p_0}^{(r)}, C_{2\ell}^{(r)}) = p_0n^{r}= n^{1+\frac{1}{2\ell-1}+o(1)}.$$
If we let $p_1= n^{-(r-2)+o(1)}$, then $\ex(G_{n,p_1}^{(r)},C_{2\ell}^{(r)})\leq n^{1+\frac{1}{2\ell-1}+o(1)}$ by the first upper bound in Theorem~\ref{thm:ourresult}. Thus by monotonicity, $\ex(G_{n,p}^{(r)}, C_{2\ell}^{(r)})= n^{1+\frac{1}{2\ell-1}+o(1)}$ for $  p\in[ p_0, p_1]$.
Note that here the choice of $p_0$ was not optimal, in particular we can get the lower bound to be  $n^{1+\frac{1}{2\ell-1}}/\omega(n)$ for any slowly growing function $\omega(n)$.

Finally, the range  of $p$ for which  our upper and lower bounds do not match up to polylog factors is $n^{-(r-2)+o(1)} \leq p \leq n^{-(r-2)+\frac{1}{2\ell-2}+o(1)} $ (the grey region in Figure 1, where the green curve is our proved upper bound, and the blue lines correspond to existing lower bounds). By monotonicity for this range of $p$ we still have $\ex(G_{n,p}^{(r)}, C_{2\ell}^{(r)})\geq n^{1+\frac{1}{2\ell-1}+o(1)}$ however if $p>  n^{-(r-2)+\frac{1}{2\ell-1}+o(1)}$ then  Proposition~\ref{prop:lowerbound2} and monotonicity imply a better lower bound $\ex(G_{n,p}^{(r)}, C_{2\ell}^{(r)}) = \Omega(pn^{r-1})$. 
\end{proof}

Note that as $\ell\rightarrow \infty$, the grey region in Figure 1  diminishes as the blue and the green lines converge to the same limit.  However,  it remains an open problem to determine the correct order of magnitude of $\ex(G_{n,p}^{(r)}, C_{2\ell}^{(r)})$ in this range of $p$ (up to polylog factors). We conjecture that the correct behaviour should be given by the lower bounds when $r=3, \ell=2$. This is explained in the final section.

Very recently Nie, Spiro and Verstra\"{e}te~\cite{jacque}  proved upper and lower bounds for $\ex(G_{n,p}^{(3)}, C_3^{(3)})$. That problem exhibits somewhat different behaviour, due to the existence of Behrend type constructions without loose triangles.
The question for odd linear cycles $C_{2\ell+1}^{(r)}$ for $\ell\geq 2$ or $\ell=1$ and $r\geq 4$ remains wide open. We believe that the behaviour for odd cycles is very different compared to our results.
\end{section}
\begin{section}{Notation}
For an $r$-graph $G$, we write $e(G)$ for its number of edges and $d(G)=r\cdot e(G)/|V(G)|$ for its average degree. If $G' \subset G$, the $r$-graph $G-G'$ has vertex set $V(G)$ and edge set $E(G)\setminus E(G')$. In particular, $e(G-G') = |E(G)\setminus E(G')|$. For an $r$-partite $r$-graph $G$ with vertex partition $(V_1, V_2, \dots, V_r)$, for $1\leq i <j \leq r$ we write $\partial_{V_i,V_j}(G)$ for the pairs $(v_i,v_j)$ with $v_i\in V_i$ and $v_j\in V_j$ such that there exists some $e\in E(G)$ with $\{v_i,v_j\}\subseteq e$. For any $r$-graph $G$,  $1\leq j\leq r$ and  any $j$-tuple $\sigma$, the \emph{degree} of $\sigma$, written $d_{G}(\sigma)$ is the number of edges that contain $\sigma$; when $\sigma=\{u,v\}$ we simplify the notation to $d_G(u,v)$. We denote by $\Delta_j(G)$ the maximum $d_G(\sigma)$ among all $j$-tuples $\sigma$.  Whenever $G$ is clear from the context we will drop it from the notation.  Given an $r$-graph $G$ and some $0<\tau<1$, the \emph{co-degree  function} $\delta(G,\tau)$  is 
$$\delta(G,\tau)=\frac{1}{d(G)}\sum_{j=2}^{r}{\frac{\Delta_{j}}{\tau^{j-1}}}.$$

 A vertex subset $I$ is called \emph{independent} in $G$ if there is no edge $e\in E(G)$ such that $e\subseteq I$. For a vertex subset $A\subseteq V(G)$ we define $G[A]$, \emph{the subgraph induced by} $A$, to be the $r$-graph with vertex set $A$ and edge set comprising all those  $e\in E(G)$ for which $e\subseteq A$. 

\end{section}
\begin{section}{Tools}

\begin{lemma}[\cite{molloy2013graph}, Chernoff bound] Given a binomially distributed variable $X\in Bin(n, p)$  and  $0<a\leq 3/2$,
 $$\Prob{[|X-\E[X]|\geq a \E[X]]}\leq 2e^{-\frac{a^2}{3}\E[X]}.$$
\end{lemma}

\begin{lemma}[\cite{alon1997nearly, kim2002asymmetry}, Azuma's Inequality for 0/1 product spaces]\label{azuma}
Let $\Omega=\{0,1\}^n$ with the $i$th coordinate of an element of $\Omega$ equal to 1 with probability $p_i$. Let $X$ be a $c$-Lipschitz random variable on $\Omega$. Set $\sigma^2=c^2\sum_{i=1}^n p_i(1-p_i)$. For all  $t\leq 2\sigma/c$, we have
$$\Prob\left(|X-\E(X)|>t\sigma \right)\leq 2e^{\frac{-t^2}{4}}$$
\end{lemma}

\begin{theorem}[\cite{MS}, Theorem 4.2] \label{thm:MSContainers}For each $r\geq 2$ there exist   $\eps_0=\eps_0(r)$ such that for all $0<\eps <\eps_0$ the following holds. Let $G$ be and $r$-graph with $N$ vertices  and suppose  $\delta(G, \tau)\leq \eps $, for some $0<\tau < 1/2$. Then there exists a collection of $\mathcal{C}\subset  \mathcal{P}(V(G))$ of at most $$exp\left(\frac{\tau N\log{1/\tau}}{\eps}\right)$$  subsets of $V(G)$ (called the containers of $G$) such that
\begin{itemize}
\item [(1)] for each independent set $I\subset V(G)$ there exists a container $C\in \mathcal{C}$ such that  $I\subseteq C$,
\item [(2)] $e(G[C])\leq (1-\eps) e(G)$, for each container $C\in \mathcal{C}$.
\end{itemize}
\end{theorem}
Below we consider the collection of copies of $H$ in $G$. This can be viewed as an $e(H)$-graph with vertex set $E(G)$, and   edge set comprising collection of edges of $G$ that form a copy of $H$. 
\begin{theorem}[\cite{MS}, Theorem 1.5]\label{thm:morrissaxton}For every $l\geq 2$, there exist  $Q >0$, $\delta >0$ and $k_0\in \mathbb{N}$ such that   the following holds for every $k\geq k_0$ and every $n\in \mathbb{N}$. Given a graph $G$ with $n$ vertices and $kn^{1+1/l}$ edges there exists a collection $\mathcal{F}$  of copies of $C_{2l}$ in $G$ satisfying 

\begin{itemize}
\item[(a)]  $|\mathcal{F}|\geq \delta k^{2l}n^2,$
\item[(b)]  $d_{\mathcal{F}}(\sigma)\leq Q  k^{2l-j-\frac{j-1}{l-1}}n^{1-\frac{1}{l}}$ for every $j$-tuple $\sigma$ for $1\leq j\leq 2l-1$. 
\end{itemize}
\end{theorem}
\end{section}
\begin{section}{Proofs}
We use the method of hypergraph containers. The main novelty of this paper is in proving a new ``balanced supersaturation" result for linear even cycles. Such a result   states roughly that an $r$-graph $G$ on $n$ vertices with significantly more than
$\ex(n, C^{(r)}_{2\ell})$  edges contains many copies of  $C_{2\ell}^{(r)}$ which are additionally distributed relatively uniformly over the edges of $G$. This was already proved in~\cite{BNS,FMS} but it was not strong enough to establish a corresponding result in random graphs as Theorem~\ref{thm:ourresult}. As we mentioned in the abstract and introduction, the output of our supersaturation result depends crucially on a ``structural" parameter of the underlying $r$-graph in addition to the density; in fact, after suitable regularizing, this is just the codegree of certain pairs of vertices (it is called $\Delta_{12}$ in the proof). Due to this dependence, our supersaturation result appears as part of the proof of the following theorem (see (P5) and (P6)), since it is quite cumbersome to separate it from the rest of the proof.

\begin{theorem}\label{thm:onestep}For every  $r\geq 3$, $\ell\geq 2$, there exist $K_0,n_0\in \mathbb{N}$ and $\eps >0$ such that the following holds for all  $n\geq n_0$ and every $K\geq K_0(\log{n})^{2r(r-1)}$. Given an $r$-graph $G$ with $n$ vertices and $Kn^{r-1}$  edges there exists a collection  $\mathcal{C}$ of at most 
 \begin{equation} \label{eqn:numcon}\exp\left(\frac{1}{\eps} \max \left \{ (\log{n})^{(\ell+3)r^2}n^{\frac{2\ell-1}{2\ell-2}}K^{-\frac{1}{2\ell-2}}, \, n^{\frac{2\ell}{2\ell-1}} (\log{n})^{2(r^2+1)} \right\}\right)\end{equation}
subgraphs of $G$ such that 

\begin{itemize} \item [(i)] every $C_{2l}^{(r)}$-free subgraph of $G$ is a subgraph of some $C\in \mathcal{C}$, 
 \item [(ii)]  $e(C)\leq  \left(1-\frac{\eps}{(\log{n})^{r^2(\ell+1) }}\right)e(G)$, for each $C\in \mathcal{C}$.
\end{itemize}
\end{theorem}

\begin{proof} The proof has two parts. In the first part we find a large collection of $2\ell$-cycles in $G$ such that the degrees of $j$-tuples in this collection, for all $1\leq j\leq 2\ell-1$, are small enough so that in the second part we are able to apply Theorem~\ref{thm:MSContainers}. Thus, we obtain  containers which contain all $C_{2\ell}$-free subgraphs of $G$. Finally, we show that the containers satisfy (i) and (ii). 

For the first part we pass  to a large $r$-partite subgraph $H'$ of $G$ with  vertex partition $(U_1,U_2, \dots U_r)$ in which all the pairs of vertices lying in different partition classes that are in some edge together, are ``almost regular''. That is, for every $1\leq i < j\leq r $ there exists some number $\Delta_{ij}$ such that if we pick $u_i\in U_i,u_j\in U_j$  such that $u_i$ and $u_j$ are in some edge together then $\Delta_{i,j}\leq d_{H'}(u_i,u_j)\leq 2\Delta_{i,j}$. Moreover, we can guarantee that no vertices in $H'$ are isolated.  Then we find a pair of partition classes, say $U_1,U_2$, with respect to which the shadow is large, that is, $\partial_{U_1U_2}(H')$ is a dense enough  $2$-graph to apply Theorem~\ref{thm:morrissaxton} and obtain a large collection of $2l$-cycles in $\partial_{U_1U_2}(H')$. Then we  expand these cycles  using the regularity of pairs of vertices in $H'$ to $2\ell$-cycles in $H'$, and thus in $G$. The  second part of the proof is more technical.  We will to show that for all $1\leq j\leq 2\ell-1$, the $j$-tuples behave well enough so that the assumptions of  Theorem~\ref{thm:MSContainers} are satisfied.

Let  $Q, \delta_0, k_0$ be obtained from Theorem~\ref{thm:morrissaxton} applied with $\ell$. Let $\eps_0$ be obtained from Theorem~\ref{thm:MSContainers} applied with $r_{\ref{thm:MSContainers}}=2\ell$. We use $a\ll b$ to mean that $a$ is sufficiently small compared to $b$ in order to apply some theorem or to satisfy some inequality in the proof. Choose constants

$$R={r\choose 2}\qquad  \alpha_r = \frac{r!}{2r^{R+r}} \qquad  \beta_r = \frac{\alpha_r}{4Rr^R} \qquad K_0=8 \,k_0^2 \, \beta_r^{-2} \qquad \delta\ll \min\{\eps_0, \, \delta_0\}\qquad \eps=\delta^4.$$ 

By a classical result of Erd\H{o}s-Kleitman~\cite{EK}, $G$ has an $r$-partite subgraph $H$ with 
$r$-partition $(V_1,V_2,\dots, V_r)$
such that $e(H)\geq r!e(G)/r^r$. Let $$D=\{\textbf{s}= \{s_{i,j}\}_{1\leq i<j\leq r } \,: s_{i,j} \in \{1, 2, \ldots,  \lfloor (r-2)\log{n} \rfloor\}\}.$$ 
For each $\textbf{s}=\{s_{i,j}\}_{1\leq i<j\leq r }\in D$, let
$$E(\textbf{s}) = \{ \{v_1,v_2,\dots, v_r\}\in E(H): v_i\in V_i \hbox{ and }
2^{s_{i,j}}\leq d_H(v_i,v_{j})<2^{s_{i,j}+1} \hbox{ for all $i,j$}\}.$$ Since $\cup_{\textbf{s}}{E(\textbf{s})}$ is a partition of  $E(H)$,   by the pigeonhole principle there exists  $\textbf{s}_0 \in D$ such that $$|E(\textbf{s}_0)|\geq \frac{e(H)}{|D|} > \frac{e(H)}{(r\log{n})^R}.$$  For every $1\leq i <j\leq r$ and $s_{i,j}\in \textbf{s}_0$, let  $\Delta_{ij}=2^{s_{i,j}}$. Let $H_0$ be the 
subgraph with   $E(H_0) = E(\textbf{s}_0)$ and  $V(H_0) =\cup_{e \in E(\textbf{s}_0)} e$. Let $U_i=V_i \cap V(H_0)$. By definition, 
\begin{equation}\label{eq:edgesofH}e(H_0)\geq \frac{e(H)}{(r\log{n})^R} \geq \frac{r!e(G)}{r^{R+r}(\log{n})^R}\geq  \frac{r!Kn^{r-1}}{r^{R+r}(\log{n})^R}.\end{equation}

Now construct a sequence of $r$-graphs $H_0\supset H_1\supset \cdots \supset H_b=H'$ as follows. For $0\le a <b$ and each pair $v_i\in U_i,v_j\in U_j$ if $$d_{H_a}(v_i,v_j)< \frac{d_{H}(v_i,v_j)}{2R(r\log{n})^R},$$ 
then delete all edges containing the pair $v_i,v_j$ from $H_a$ and let the resulting $r$-graph be $H_{a+1}$.  By (\ref{eq:edgesofH}), the number of edges deleted during this process is at most 
$$\sum_{1\leq i<j\leq r}\sum_{v_i\in V_i, v_j\in V_j}{\frac{d_H(v_i,v_j)}{2R(r\log{n})^R}} \leq \frac{e(H)}{2(r\log{n})^R} \leq \frac{e(H_0)}{2}.$$
Consequently, the process terminates with $H'$ and $e(H') \ge e(H_0)/2$.  Again using (\ref{eq:edgesofH}) we obtain
\begin{equation}\label{eq:edgesofHprime}
e(H')\geq  \frac{e(H_0)}{2} \ge   \frac{r!Kn^{r-1}}{2r^{R+r}(\log{n})^R} =  \frac{\alpha_r  Kn^{r-1}}{(\log{n})^R}.
\end{equation}
We also delete all the vertices which became isolated in $U_i$ for all $1\leq i \leq r$. 
For simplicity of presentation,  we keep using the letters $U_i$ for the parts of $H'$.  Let $\partial_{i,j}=\partial_{U_i,U_j}(H')$. Now $H'$ satisfies the following, for all $1\leq i<j\leq r$:
\begin{itemize}
\item  [(P1)] there are no isolated vertices in $H'$
\item  [(P2)] if  $v_i\in U_i, v_j\in U_j$  such that there exists $e\in E(H')$ with $\{v_iv_j\}\subseteq e$  then  $$\frac{\Delta_{ij}}{2R(r\log{n})^R}\leq d_{H'}(v_i,v_j) \leq 2\Delta_{ij}.$$
\end{itemize}

Without loss of generality, we may assume $U_1$ is the largest among $U_1, U_2, \dots, U_r$. 

\begin{claim} \label{claim5.3} Either $|\partial_{12}| \geq \sqrt{K} |U_1|^{3/2}$ or  $|\partial_{13}| \geq  \frac{\beta_r\sqrt{K}|U_1|^{3/2}}{(\log{n})^{2R}}$.
\end{claim}
\begin{proof} Suppose $|\partial_{12}| < \sqrt{K} |U_1|^{3/2}$.  Then by (\ref{eq:edgesofHprime})

\begin{equation} \frac{\alpha_r  Kn^{r-1}}{(\log{n})^R} \leq e(H') \leq  |\partial_{12}| (2\Delta_{12})\leq 2\sqrt{K} |U_1|^{3/2}\Delta_{12}.
\end{equation}
It follows that $$\Delta_{12} \geq \frac{\alpha_r\sqrt{K}n^{r-1}}{2|U_1|^{3/2}(\log{n})^R}.$$ 
By (P1), for every vertex $u_1\in U_1$ there is some vertex $u_2\in U_2$ such that $u_1$ and $u_2$ are in an edge of $H'$.  As $d_{H'}(\{u_1, u_2, u_3\}) \le n^{r-3}$ for each $u_3 \in U_3$, the number of distinct vertices $u_3\in U_3$ such that $\{u_1,u_2,u_3\}$ is in some edge  of $H'$ is at least $d_{H'}(u_1,u_2)/n^{r-3}$. In other words, $d_{\partial_{13}}(u_1) \ge d_{H'}(u_1,u_2)/n^{r-3}$. Therefore, by (P2),

$$|\partial_{13}|  =\sum_{u_1 \in U_1} d_{\partial_{13}}(u_1) \ge  \frac{|U_1|\Delta_{12}}{2R(r\log{n})^Rn^{r-3}} \geq \frac{\alpha_r\sqrt{K}n^{2}}{4Rr^R |U_1|^{1/2}(\log{n})^{2R}}\geq \frac{\beta_r\sqrt{K}|U_1|^{3/2}}{(\log{n})^{2R}}$$
which proves the claim.
\end{proof}

By Claim~\ref{claim5.3}, we may assume  without loss of generality that

\begin{equation}\label{eq:partialedges}|\partial_{12}| \geq  \frac{\beta_r\sqrt{K}|U_1|^{3/2}}{(\log{n})^{2R}} \geq \frac{\beta_r\sqrt{K/8} \, m^{3/2}}{(\log{n})^{2R}},
\end{equation} 
where $m=|U_1|+|U_2|$.

\begin{claim}\label{claim:degrees}  $\Delta_{12}\geq 8\ell r^{R+1} (\log{n})^R n^{r-3}$. 
\end{claim}
\begin{proof}
By $(\ref{eq:edgesofHprime})$ and (P2),
$$ \frac{\alpha_r  Kn^{r-1}}{(\log{n})^R }\leq e(H')\leq |\partial_{12}| (2\Delta_{12}) \leq 2n^2 \Delta_{12}.$$ 
Also $K \geq K_0 (\log n)^{4R} \ge 8k_0^2\beta_r^{-2} (\log{n})^{4R}$. Hence  
$$ \Delta_{12} \geq \frac{\alpha_r  Kn^{r-3}}{2(\log{n})^R}  
\ge \frac{\alpha_r  8k_0^2\beta_r^{-2} (\log{n})^{4R}n^{r-3}}{2(\log{n})^R} =
   4\alpha_r \beta_r^{-2} k_0^2 (\log{n})^{3R}n^{r-3}
\ge 
8\ell R r^{R+1} (\log{n})^R n^{r-3},$$
where the last inequality follows since $n$ is sufficiently large compared to $\ell, r, k_0$.
\end{proof}

Set $k=\frac{|\partial_{12}|}{m^{1+1/\ell}}$. By (\ref{eq:partialedges}) and $K \geq 8k_0^2\beta_r^{-2} (\log{n})^{2r(r-1)}$,
$$k = \frac{|\partial_{12}|}{m^{1+1/\ell}} \ge
\frac{\beta_r\sqrt{K/8} \, m^{3/2}}{(\log{n})^{2R}m^{1+1/\ell}} \ge 
\frac{\beta_r \, (k_0\beta_r^{-1} (\log n)^{2R})\, m^{3/2}}{(\log n)^{2R} m^{1+1/\ell}} = k_0m^{\frac{1}{2} - \frac{1}{\ell}} \geq k_0.$$ 

Therefore, by Theorem~\ref{thm:morrissaxton} applied with $k$ and $m$, the  shadow graph $\partial_{12}$ contains a collection $\mathcal{F}$  of copies of $C_{2\ell}$ satisfying 

\begin{itemize}
\item[(P3)]  $|\mathcal{F}|\geq \delta k^{2\ell}m^2 ,$
\item[(P4)] $\Delta_j(\mathcal{F})\leq Q  k^{2l-j-\frac{j-1}{\ell-1}}m^{1-\frac{1}{\ell}}$ for all $1\leq j\leq 2\ell-1$. 
\end{itemize}

Let $C$ be any $2\ell$-cycle included in  $\mathcal{F}$ with consecutive edges $x_1x_2\dots x_{2\ell}$ in the natural cyclic ordering. Since  $d_{H'}(x_i,x_{i+1})\geq \Delta_{12}/2R(r\log{n})^R$ for every $1\leq i <j \leq \ell$, it follows  by (P2) that the number of ways to extend $C$ to some linear $2\ell$-cycle in $H'$ is at least
 $$\frac{\Delta_{12}}{2R(r\log{n})^R}\left(\frac{\Delta_{12}}{2R(r\log{n})^R}-(r-2)n^{r-3}\right)\dots \left(\frac{\Delta_{12}}{2R(r\log{n})^R}-(2\ell-1)(r-2)n^{r-3}\right).$$
By Claim~\ref{claim:degrees}, this is at least  $$\left(\frac{\Delta_{12}}{4R(r\log{n})^R}\right)^{2\ell}. $$
Let $Ext(C)$ be the collection of all cycles $C_{2\ell}^{(r)}$ obtained from $C$ in this manner. Let $\mathcal{F}' =\{Ext(C)|C\in \mathcal{F}\}$, so $\mathcal{F}'$ is collection of edge sets of linear $C_{2\ell}^{(r)}$ in $H'$ (and, as in Theorem~\ref{thm:morrissaxton}, $\mathcal{F}'$ can also be viewed as a $2\ell$-graph). By the previous discussion and (P3) and (P4):
\begin{itemize}
\item[(P5)]  $|\mathcal{F}'|\geq \delta k^{2l}m^2  \left(\frac{\Delta_{12}}{4R(r\log{n})^R}\right)^{2\ell} ,$
\item[(P6)]  $\Delta_j(\mathcal{F}') \leq Q  k^{2\ell-j-\frac{j-1}{\ell-1}}m^{1-\frac{1}{\ell} } (2\Delta_{12})^{2\ell -j}$ for  all $1\leq j\leq 2l-1$.
\end{itemize}

Now let ${S}$ be the $2\ell$-graph with $$V(S)=E(G) \qquad \hbox{  and } \qquad E(S) = {E(G) \choose 2\ell} \cap \mathcal{F}'.
$$
In other words, vertices of $S$ are edges of $G$ and edges of $S$ are copies of $C_{2\ell}^{(r)}$  in $\mathcal{F}'$.  The edges in $E(G)\setminus E(H')$ are isolated vertices in ${S}$, and $ \Delta_j(S)\leq \Delta_j(\mathcal{F}') $. Therefore, (P5) and (P6) translate to the following properties for ${S}$:
\begin{itemize}
\item[(P7)]  $e(S)\geq \delta k^{2\ell}m^2  \left(\frac{\Delta_{12}}{4R(r\log{n})^R}\right)^{2\ell} ,$
\item[(P8)]   $\Delta_j(S) \leq Q  k^{2\ell-j-\frac{j-1}{\ell-1}}m^{1-\frac{1}{\ell}}  (2\Delta_{12})^{2\ell -j}$ for  all $1\leq j\leq 2\ell-1$.
\end{itemize}

\begin{claim}\label{claim:partialdegrees}$  |\partial_{12}| \Delta_{12} \geq \frac{\alpha_r  Kn^{r-1}}{2(\log{n})^R} $.
\end{claim}
\begin{proof} By (\ref{eq:edgesofHprime}) and (P2), 
$$\frac{\alpha_r  Kn^{r-1}}{(\log{n})^R} \leq e(H') \leq  |\partial_{12}| (2\Delta_{12}),$$ and the claim follows.
\end{proof}

Now let us compute the co-degree  function $\delta(S,\tau)$, for the following choice of $\tau$: 
\begin{equation}\label{eq:tau}\tau=\max\left\{\delta^{-3} (\log{n})^{r^2(\ell+2)} K^{-\frac{2\ell-1}{2\ell-2}}n^{-(r-2)+\frac{1}{2\ell-2}}, \, \delta^{-3} (\log{n})^{2r^2}K^{-1} n^{-(r-2)+\frac{1}{2\ell-1}} \right\}.\end{equation}

By (\ref{eq:edgesofH}),
$$e(G)\leq \frac{r^{R+r}(\log{n})^R}{r!}e(H_0) \le
\frac{2r^{R+r}(\log{n})^R}{r!}e(H')
\leq  \frac{2r^{R+r}(\log{n})^R}{r!} |\partial_{12}| (2\Delta_{12}),$$
and it follows that 
\begin{equation}\label{eq:avgdegreeS}d(S)=\frac{2\ell e(S)}{e(G)} \geq \frac{ 2\delta \ell k^{2\ell}m^2  \left(\frac{\Delta_{12}}{4R(r\log{n})^R}\right)^{2\ell}  }{ \frac{4r^{R+r}}{r!}(\log{n})^R |\partial_{12}|\Delta_{12}} \geq \frac{\delta k^{2\ell-1} \Delta_{12}^{2\ell-1}m^{1-1/\ell} }{(\log{n})^{(2\ell +1)R+1}}.
\end{equation}

where in the  second inequality we used $|\partial_{12}|=km^{1+1/\ell}$ and that $n$ is sufficiently large.

By (\ref{eq:avgdegreeS}) and (P8) for all $2\leq j \leq 2\ell -1$, we have 

\begin{align*}\left(\frac{\Delta_j(S)}{d(S)}\right)^{1/(j-1)}\tau^{-1}&\leq   \left(\frac{Qk^{2l-j-\frac{j-1}{l-1}} (2\Delta_{12})^{2\ell -j}(\log{n})^{(2\ell +1) R+1}}{\delta k^{2\ell-1} \Delta_{12}^{2\ell-1}}\right)^{1/(j-1)}\tau^{-1}\\
&\leq  \left(\delta^{-1} (\log{n})^{(2\ell+1) R+2} k^{-(j-1)-\frac{j-1}{l-1}} (\Delta_{12})^{-(j-1)}\right)^{1/(j-1)}\tau^{-1}\\
&\leq  \delta^{-1} (\log{n})^{(2\ell+1) R+2} k^{-1-\frac{1}{\ell-1}} (\Delta_{12})^{-1} \tau^{-1} \\
&=   \delta^{-1}\tau^{-1}(\log{n})^{(2\ell+1) R+2} \times \frac{m^{1 +\frac{2}{\ell-1}}}{(|\partial_{12}|\Delta_{12}) |\partial_{12}|^\frac{1}{\ell-1}} \\
&\leq  \delta^{-1}\tau^{-1}(\log{n})^{(2\ell+4)R+3} \times \frac{m^{1+\frac{2}{\ell-1} -\frac{3}{2(\ell-1)}}}{K^{\frac{2\ell-1}{2\ell-2}}n^{r-1}} \\
&\leq \delta^{-1}\tau^{-1}(\log{n})^{r^2(\ell+2)} K^{-\frac{2\ell-1}{2\ell-2}}n^{-(r-2)+\frac{1}{2\ell-2}} \\
&<\frac{\delta}{2\ell}.
\end{align*}

Above in the first equality we used  that $|\partial_{12}|=km^{1+1/\ell}$, in the fourth inequality we used (\ref{eq:partialedges}) and  Claim~\ref{claim:partialdegrees}, and in the last inequality we used the definition of $\tau$ and $\delta <1/2\ell$.

As for $j=2\ell$, since $\Delta_{2\ell}\leq 1$, it follows
\begin{align*}\left(\frac{\Delta_{2\ell}(S)}{d(S)}\right)^{1/(2\ell-1)} \tau^{-1}&\leq \left(\frac{(\log{n})^{(2\ell +1)R +1}}{\delta k^{2\ell-1} \Delta_{12}^{2\ell-1}m^{1-1/\ell} }\right)^{1/(2\ell-1)}\tau^{-1} \\
&\leq \tau^{-1}\delta^{-1} \frac{(\log{n})^{2 R}m^{1+\frac{1}{\ell}-\frac{\ell-1}{\ell(2\ell-1)}}}{(|\partial_{12}|\Delta_{12})}\\
&\leq \tau^{-1}\delta^{-1} (\log{n})^{3R+1}K^{-1} n^{-(r-2)+\frac{1}{2\ell-1}}\\
&<\frac{\delta}{2\ell},
\end{align*}
where in the second inequality we used  $|\partial_{12}|=km^{1+1/\ell}$ and in the third inequality we used   Claim~\ref{claim:partialdegrees}. So for all $2\leq j \leq 2\ell $,

\begin{equation}\label{ineq:degrees} \left(\frac{\Delta_j(S)}{d(S)}\right)\tau^{-(j-1)}<\left(\frac{\delta}{2\ell}\right)^{j-1}\leq \left(\frac{\delta}{2\ell}\right).\end{equation} 
By (\ref{ineq:degrees}),
$$\delta(S,\tau) = \frac{1}{d(S)}\sum_{j=2}^{2\ell}{\frac{\Delta_{j}(S)}{\tau^{j-1}}}<\delta,$$ and we can therefore apply  Theorem~\ref{thm:MSContainers} to $S$ and obtain a collection of $\mathcal{C}\subseteq  \mathcal{P}(V(S)) = \mathcal{P}(E(G))  $ of at most $$exp\left(\frac{\tau |V(S)|\log{1/\tau}}{\delta}\right)$$  subsets of $V(S)$ such that
\begin{itemize}
\item [(P9)] for each independent set $I$ in  $S$ there exists a ``container'' $C\in \mathcal{C}$ such that  $I\subseteq C$,
\item [(P10)] $e(S[C])\leq (1-\delta) e(S)$, for each container $C\in \mathcal{C}$.
\end{itemize}

By the choice of $\tau$,
$$\left(\frac{\tau |V(S)|\log{1/\tau}}{\delta}\right) \leq \frac{1}{\eps} \max \left \{ (\log{n})^{(\ell+3)r^2}n^{\frac{2\ell-1}{2\ell-2}}K^{-\frac{1}{2\ell-2}}, \, n^{\frac{2\ell}{2\ell-1}} (\log{n})^{2(r^2+1)} \right\}$$
and this proves (\ref{eqn:numcon}).
 We will now prove statements (i) and (ii) of the theorem.

Each $C_{2\ell}^{(r)}$-free subgraph of $G$ is an  independent set  of $S$, and each $C\in \mathcal{C}$ can be viewed as  a subgraph of $G$. So (P9) implies that for each $C_{2\ell}^{(r)}$-free subgraph $G'$ of $G$  there exists a container $C\in \mathcal{C}$ such that  $G'\subseteq C$. This shows that (i) is true. To conclude, it remains to show that (P10) implies that statement (ii) of the theorem holds.

Let   $C\in \mathcal{C}$. Recall that by definition, every vertex $v\in V(S)$ is an edge in $E(G)$. For each $e\in E(S-S[C])$ there exists  $v\in e$ such that $v\in V(S)\setminus C= E(G)\setminus E(C)$. Since  $d_{S}(v)\leq \Delta_1(S)$ for every $v \in V(S)$,
\begin{equation}\label{eq:containers} e(S-S[C]) \leq \Delta_1(S) e(G-C).
\end{equation}

On the other hand,  (P7) and (P8) and $\delta$ sufficiently small imply that
\begin{align*}
\frac{e(S)}{\Delta_1(S)} &\geq \frac{\delta k^{2l}m^2  \left(\frac{\Delta_{12}}{4R(r\log{n})^R}\right)^{2\ell}}{Q  k^{2\ell-1}m^{1-\frac{1}{\ell}}  (2\Delta_{12})^{2\ell -1}}\\
&\geq  \frac{2}{Q(8Rr^{R})^{2\ell} } \cdot \frac{\delta km^{1+\frac{1}{\ell}}\Delta_{12}}{(\log{n})^{2\ell R}} \\
&\geq  \frac{\delta^2 km^{1+\frac{1}{\ell}}\Delta_{12}}{(\log{n})^{2\ell R}}.
\end{align*}
 Thus, together with (\ref{eq:containers})  and (P10) it follows that  $$\delta \, e(S) \leq e(S- S[C])\leq \Delta_1(S) \, e(G-C)\leq e(S) \,  \frac{(\log{n})^{2\ell R}}{\delta^2 km^{1+\frac{1}{\ell}}\Delta_{12}} \, e(G-C).$$
Therefore, 

$$e(G-C) \geq \frac{\delta^3 k\Delta_{12} m^{1+\frac{1}{\ell}}}{ (\log{n})^{2\ell R}}\geq  \frac{ \delta^3 (|\partial_{12}| \Delta_{12})}{ (\log{n})^{2\ell R}} \geq \frac{\delta^3 \alpha_r Kn^{r-1}}{2(\log{n})^{(2\ell+1) R}}\geq \frac{\delta^4 e(G)}{(\log{n})^{(2\ell+1) R}}\geq  \frac{\eps \, e(G)}{(\log{n})^{r^2(\ell+1)}} ,$$   where in the second inequality we used the definition of $k$, in the third one we used Claim~\ref{claim:partialdegrees}, in the fourth that $\delta$ is sufficiently small, and the last one we use $\eps=\delta^4$. As $e(C) = e(G)-e(G-C)$, the proof of (ii) is complete. \end{proof}

\begin{theorem}\label{thm:count} For every $r\geq 3$, $\ell\geq 2$ there exist  $C=C(\ell)$, $K_0\in \mathbb{N}$ such that the following holds for all sufficiently large $n\in \mathbb{N}$ and $K_0(\log{n})^{2r(r-1)}\leq K\leq n^{1/(2\ell-1)} (\log{n})^{2\ell r^2(\ell-1)}$.  There exists a collection $\mathcal{G}_{\ell,r}(n,K)$ of at most 

$$exp\left( C   n^{\frac{2\ell-1}{2\ell-2}}K^{-\frac{1}{2\ell-2}} (\log{n})^{(\ell+3)r^2+1}\right)$$
 $r$-graphs on vertex set $[n]$ such that  $e(G)\leq Kn^{r-1}$ for each $G\in \mathcal{G}_{\ell,r}(n,K)$,  and every $C_{2\ell}^{(r)}$-free $r$-graph on $[n]$ is a subgraph of some $G\in \mathcal{G}_{\ell}(n,K)$. 
\end{theorem}
\begin{proof} Let $K_0$ and $\eps$ be obtained from Theorem~\ref{thm:onestep} applied with $\ell$. We apply Theorem~\ref{thm:onestep} iteratively, each time refining the set of
containers obtained at the previous step.   We start with $\mathcal{C}_0=\left\{K_n^{(r)}\right\}$. For $i\geq 1$, let $$K_{i}=\max\left\{\left(1-\frac{\eps}{(\log{n})^{r^2(\ell+1)}}\right)^in,K_0(\log{n})^{2r(r-1)}\right\}.$$ At step $t$  we obtain a collection $\mathcal{C}_t$ of $r$-graphs on $[n]$ such that $e(G)\leq K_tn^{r-1}$ for every $G\in \mathcal{C}_t$  and every $C_{2\ell}^{(r)}$-free graph on $[n]$ is a subgraph of some $G\in \mathcal{C}_t$, and moreover,
\begin{equation}\label{count:total}|\mathcal{C}_t|\leq exp\left(\frac{1}{\eps} \sum_{i=1}^{t}  \max  \left \{ (\log{n})^{(\ell+3)r^2}n^{\frac{2\ell-1}{2\ell-2}}K_i^{-\frac{1}{2\ell-2}}, \, n^{\frac{2\ell}{2\ell-1}} (\log{n})^{2(r^2+1)} \right\} \right).
\end{equation}

For $i\geq 0$, at  step $i+1$, we apply Theorem~\ref{thm:onestep} to each graph $G\in \mathcal{C}_i$ with $e(G)\geq K_{i+1}n^{r-1}$  (note that  $e(G)\leq K_{i}n^{r-1}$, as otherwise $G$ would not have been in $ \mathcal{C}_i$) and obtain a family  $\mathcal{C}(G)$ of subgraphs of $G$ of with
\begin{equation}\label{count:containers}
|\mathcal{C}(G)| \le 
exp\left(\frac{1}{\eps} \max  \left \{ (\log{n})^{(\ell+3)r^2}n^{\frac{2\ell-1}{2\ell-2}}K_{i+1}^{-\frac{1}{2\ell-2}}, \, n^{\frac{2\ell}{2\ell-1}} (\log{n})^{2(r^2+1)} \right\} \right).
\end{equation}
The family $\mathcal{C}(G)$  satisfies the following properties:
\begin{itemize} \item [(a)] every $C_{2l}^{(r)}$-free subgraph of $G$ is a subgraph of some $C\in \mathcal{C}(G)$, 
 \item [(b)]   For each $C\in \mathcal{C}(G)$, $$e(C)\leq \left(1-\frac{\eps}{(\log{n})^{r^2(\ell+1)}}\right)e(G)\leq    \left(1-\frac{\eps}{(\log{n})^{r^2(\ell+1)}}\right) K_in^{r-1} \leq K_{i+1}n^{r-1}.$$
\end{itemize}

If $e(G)\leq K_{i+1}n^{r-1}$  we let $\mathcal{C}(G)=\{G\}$. We define $\mathcal{C}_{i+1}=\cup_{G\in \mathcal{C}_i}{\mathcal{C}(G)}.$ Let $m$ be the minimum such that $K_m\leq K$.  We iterate until we obtain $\mathcal{C}_m$. It is easy to check that  $m=O(\log{n})$. This allows us to get the desired bound on the cardinaly of $\mathcal{C}_m$. Indeed, for all those $0\leq i \leq m-1$ for which   the maximum in (\ref{count:containers}) is obtained as the second term of the expression, in total their contribution to the exponent in (\ref{count:total}) for $\mathcal{C}_m$ is at most

\begin{equation}\label{eq:upperbound1}mn^{\frac{2\ell}{2\ell-1}}(\log{n})^{2(r^2+1)} \leq n^{\frac{2\ell-1}{2\ell-2}}K^{-\frac{1}{2\ell-2}} (\log{n})^{(\ell+3)r^2} .
\end{equation}

For all those $0\leq i\leq  m-1$ for which   the maximum in (\ref{count:containers}) is obtained as the first term of the expression, since $K_i>K$, their total contribution to the exponent in (\ref{count:total}) for $\mathcal{C}_m$ is at most
 
\begin{equation}\label{eq:upperbound2}mn^{\frac{2\ell-1}{2\ell-2}}K^{-\frac{1}{2\ell-2}} (\log{n})^{(\ell+3)r^2} \leq  O(1) n^{\frac{2\ell-1}{2\ell-2}}K^{-\frac{1}{2\ell-2}} (\log{n})^{(\ell+3)r^2+1}.
\end{equation}

Finally for $i=m$, since $K_{m} \le K \le K_{m-1}$, 
$$ \left(1-\frac{\eps}{(\log{n})^{r^2(\ell+1)}}\right) K \leq  \left(1-\frac{\eps}{(\log{n})^{r^2(\ell+1)}}\right) K_{m-1} \leq K_{m} \le  K.$$ Together with $K \leq n^{1/(2\ell-1)} (\log{n})^{2\ell r^2(\ell-1)}$ this implies
\begin{align*}\max \left \{ n^{\frac{2\ell-1}{2\ell-2}}K_m^{-\frac{1}{2\ell-2}} (\log{n})^{(\ell+3)r^2}, \, n^{\frac{2\ell}{2\ell-1}}(\log{n})^{2(r^2+1)}\right\} &=  n^{\frac{2\ell-1}{2\ell-2}}K_m^{-\frac{1}{2\ell-2}} (\log{n})^{(\ell+3)r^2}\\
&=  O(1)  n^{\frac{2\ell-1}{2\ell-2}}K^{-\frac{1}{2\ell-2}} (\log{n})^{(\ell+3)r^2} \numberthis \label{eq:upperbound3}.\end{align*}

Thus, putting (\ref{eq:upperbound1}), (\ref{eq:upperbound2}),  (\ref{eq:upperbound3})  all together we obtain
\begin{align*}
|\mathcal{C}_m| &\leq  \sum_{i=1}^{m}  \max \left \{ n^{\frac{2\ell-1}{2\ell-2}}K_i^{-\frac{1}{2\ell-2}} (\log{n})^{(\ell+3)r^2}, n^{\frac{2\ell}{2\ell-1}}(\log{n})^{2(r^2+1)}\right\} \\
&=  O(1) n^{\frac{2\ell-1}{2\ell-2}}K^{-\frac{1}{2\ell-2}} (\log{n})^{(\ell+3)r^2+1}.
\end{align*}
To finish the proof, we let $\mathcal{G}_{\ell,r}(n,K)=\mathcal{C}_m$.
\end{proof}

\begin{theorem}\label{thm:main} Fix $r\geq 3$ and $\ell\geq 2$. Set 
$$ p_0=n^{-(r-2)}(\log{n})^{-(2\ell-1)\ell r^2} \quad \hbox{ and } \quad p_1   = n^{-(r-2)+\frac{1}{2\ell-2}} (\log{n})^{-3r(r-1)}.$$
Then a.a.s.
$$\ex\left(G_{n,p}, C_{2\ell}^{(r)}\right)\leq \begin{cases} p^{\frac{1}{(2\ell-1)}}n^{1+\frac{r-1}{2\ell-1}}(\log{n})^{(\ell+3)r^2+2}, & \mbox{if } p_0 \le  p\leq p_1\\ 
pn^{r-1} (\log{n})^{(\ell+3)r^2+1},& \mbox{$p>p_1$.} \end{cases}$$
\end{theorem}

\begin{proof}
 Let $C$, $K_0$ and $n_0$  be derived from Theorem~\ref{thm:count} such that the statement holds for all $n\geq n_0$ and $K\geq K_0(\log{n})^{2r(r-1)}$. First let us assume $p_0\le p\leq p_1$. Choose $K_1$ such that $p=K_1^{-\frac{2\ell-1}{2\ell-2}}n^{-(r-1)+\frac{2\ell-1}{2\ell-2}}$. Because $p~\geq p_0$, 
$$K_1=p^{-\frac{2\ell-2}{2\ell-1}}n^{-\frac{(r-1)(2\ell-2)}{2\ell-1}+1}\leq  n^{\frac{1}{2\ell-1}} (\log{n})^{2\ell r^2(\ell-1)}.$$

On the other hand, since  $p\leq p_1$,  $$K_1 = p^{-\frac{2\ell-2}{2\ell-1}}n^{-(r-1)\frac{(2\ell-2)}{2\ell-1}+1} \geq (\log{n})^{3r(r-1)\frac{2\ell-2}{2\ell-1}}\geq K_0(\log{n})^{2r(r-1)}.$$
Thus we can apply Theorem~\ref{thm:count} with parameters $\ell, r, K_1$ and obtain a family  $ \mathcal{G}_{\ell,r}(n,K_1)$ of $r$-graphs on  vertex set $[n]$  such that every $r$-graph $G\in \mathcal{G}_{\ell,r}(n,K_1)$ satisfies $e(G)\leq K_1n^{r-1}$ and moreover, every $C_{2\ell}^{(r)}$-free $r$-graph on $[n]$ is a subgraph of some $G\in \mathcal{G}_{\ell,r}(n,K_1)$. By Theorem~\ref{thm:count},

$$ |\mathcal{G}_{\ell,r}(n,K_1)|\leq exp\left( C  n^{\frac{2\ell-1}{2\ell-2}}K_1^{-\frac{1}{2\ell-2}} (\log{n})^{(\ell+3)r^2+1} \right).$$ 
Set $m=p^{\frac{1}{(2\ell-1)}}n^{1+\frac{r-1}{2\ell-1}}(\log{n})^{(\ell+3)r^2+2}$ and suppose that $G_{n,p}$ has a $C_{2\ell}^{(r)}$-free subgraph $H$ with $m$ edges. Then there exists some $G\in \mathcal{G}_{\ell,r}(n,K_1)$ which contains $H$, in particular $G$ contains at least $m$ edges of $G_{n,p}$. Therefore, 
the expected number of such subgraphs $H$ is at most
$$ |\mathcal{G}_{\ell, r}(n,K_1)| {K_1n^{r-1} \choose m} p^m.$$ 
The choice of $K_1$ and $p\leq p_1$
imply that
$$m \gg 
\max \{Cn^{\frac{2\ell-1}{2\ell-2}}K_1^{-\frac{1}{2\ell-2}} (\log{n})^{(\ell+3)r^{2}+1}, pK_1n^{r-1} \} =\max \{\log(\mathcal{G}_{\ell, r}(n, K_1)), pK_1n^{r-1} \}.$$ 
Consequently, 
\begin{align*}
    |\mathcal{G}_{\ell,r}(n,K_1)| {K_1n^{r-1} \choose m} p^m \leq   \left(O(1) \cdot \frac{pK_1n^{r-1}}{m}\right)^m\longrightarrow 0.
\end{align*}

Now let us assume  $p> p_1$ and let $K_2=K_0(\log{n})^{(\ell+3)r^2}$. By Theorem~\ref{thm:count} there exists a collection $\mathcal{G}_{\ell,r}(n,K_2)$ of $r$-graphs on   vertex set $[n]$ such that $e(G)\leq K_2n^{r-1}$ for every $G\in \mathcal{G}_{\ell,r}(n,K_2)$ and such that  every $C_{2\ell}^{(r)}$-free $r$-graph on $[n]$ is a subgraph of some $G\in \mathcal{G}_{\ell, r}(n,K_2)$. Furthermore,
$$|\mathcal{G}_{\ell,r}(n,K_2)|\leq exp\left( C  n^{\frac{2\ell-1}{2\ell-2}}K_2^{-\frac{1}{2\ell-2}} (\log{n})^{(\ell+3)r^2+1} \right).$$ 

Set $m=pn^{r-1} (\log{n})^{(\ell+3)r^2+1}$ and suppose that $G_{n,p}$ has a $C_{2\ell}^{(r)}$-free subgraph $H$ with $m$ edges. Then there exists some $G\in \mathcal{G}_{\ell,r}(n,K_2)$ which contains $H$, in particular $G$ contains at least $m$ edges of $G_{n,p}$. Therefore, 
the expected number of such subgraphs $H$ is at most
$$ |\mathcal{G}_{\ell,r }(n,K_2)| {K_2n^{r-1} \choose m} p^m.$$ 
The choice of $K_2$ and $p >p_1$ imply that
$$m \gg 
\max \{C n^{\frac{2\ell-1}{2\ell-2}}K_2^{-\frac{1}{2\ell-2}}(\log{n})^{(\ell+3)r^2+1}, \,
pK_2n^{r-1} \} =\max \{\log(\mathcal{G}_{\ell, r}(n, K_2)), \, pK_2n^{r-1} \}.$$ 
Consequently, 
\begin{align*}
    |\mathcal{G}_{\ell, r}(n,K_2)| {K_2n^{r-1} \choose m} p^m \leq   \left(O(1) \cdot \frac{pK_2n^{r-1}}{m}\right)^m\longrightarrow 0,
\end{align*}
and the proof is complete.
\end{proof}

\subsection{Lower bounds}
\begin{proposition}\label{prop:lowerbound1} For every $\ell\geq 2, r\geq 3$, if $n^{-r}\ll p\ll  n^{-(r-1)+\frac{1}{2\ell-1}}$, then a.a.s. $G_{n,p}^{(r)}$ contains a $C_{2\ell}^r$-free subgraph with $(1-o(1))p{n \choose r}$  edges.
\end{proposition}

\begin{proof} Note that $\E[e(G_{n,p}^{(r)})]= p{n \choose r} $.  Let ${X}$ denote the number of copies of $C_{2\ell}^{(r)}$ in $G^{(r)}_{n,p}$. Then $$\E[{X}]=O(n^{2\ell(r-1)}p^{2\ell}).$$
Let   $\omega(n) = \frac{\E[e(G^{(r)}_{n,p})]}{\E[X]}$. By assumption on $p$, we have that $\omega(n)\rightarrow \infty$ as $n\rightarrow \infty$. Let $\eps(n)$ by any function such that $\eps(n)\rightarrow0$ as $n\rightarrow \infty$ but such that $\eps(n)\omega(n)\rightarrow \infty$. Then by Markov's inequality $$\Prob\left[X\geq \eps(n) p {n \choose r }\right]\leq \frac{1}{\eps(n) \omega(n) } \longrightarrow 0.$$  Thus, with high probability $X=o(  p{n \choose r})=o(\E[e(G_{n,p}^{(r)})])$. On the other hand, it is a well known fact that for $p\gg n^{-r}$, the random variable $e(G^{(r)}_{n,p})$ is a binomial variable concentrated around its mean, hence with high probability $X=o(e(G^{(r)}_{n,p}))$. Therefore, with high probability, by deleting one edge from each copy of  $C_{2\ell}^{(r)}$ we will obtain a $C_{2\ell}^{(r)}$-free subgraph of $G_{n,p}^{(r)}$ with $(1-o(1))e(G_{n,p}^{(r)})$  edges. 
\end{proof}

\begin{proposition}\label{prop:lowerbound2} For every $\ell\geq 2, r\geq 3$, if $p\gg n^{-(r-1)}$, then a.a.s. $G_{n,p}^{(r)}$ contains a $C_{2\ell}^{(r)}$-free subgraph with $(1-o(1))p{n-1 \choose r-1}$  edges.
\end{proposition}
\begin{proof} For any vertex $v$, we have $\E[d(v)]=p { n -1 \choose r-1}$ thus for $p\gg n^{-(r-1)}$ by Chernoff's inequality with high probability, $d(v)=(1+o(1))p { n -1 \choose r-1}$. Therefore, by letting $H$ be the subgraph of $G_{n,p}^{(r)}$ comprising all the edges containing a fixed  vertex $v$ we obtain a $C_{2\ell}^{(r)}$-free subgraph of the required size. 
\end{proof}
\end{section}

\begin{section}{Concluding remarks and some constructions}
It remains an open problem to determine the full behaviour of $\ex(G_{n,p}^{(r)}, C_{2\ell}^{(r)})$. To improve our upper bound for the range  $n^{-(r-2)+o(1)} \leq p \leq n^{-(r-2)+\frac{1}{2\ell -2}}$ using hypergraph containers one would need to improve Theorem~\ref{thm:onestep}, in particular, a better balanced supersaturation result needs to be obtained. Note that to prove our balanced supersaturation result ((P5), (P6)) first we find  a corresponding collection of $2$-cycles in some shadow of the $r$-graph ((P3),(P4)) using Theorem~\ref{thm:morrissaxton}, and then we extend these $2$-cycles to $r$-cycles in the original $r$-graph. We want to emphasize that Theorem~\ref{thm:morrissaxton}  is tight, so to improve our balanced supersaturation result one needs  to find the collection of well-behaving collection of $r$-cycles in the $r$-graph directly.

\begin{figure}[H]
  \centering
    \includegraphics[width=.8\textwidth]{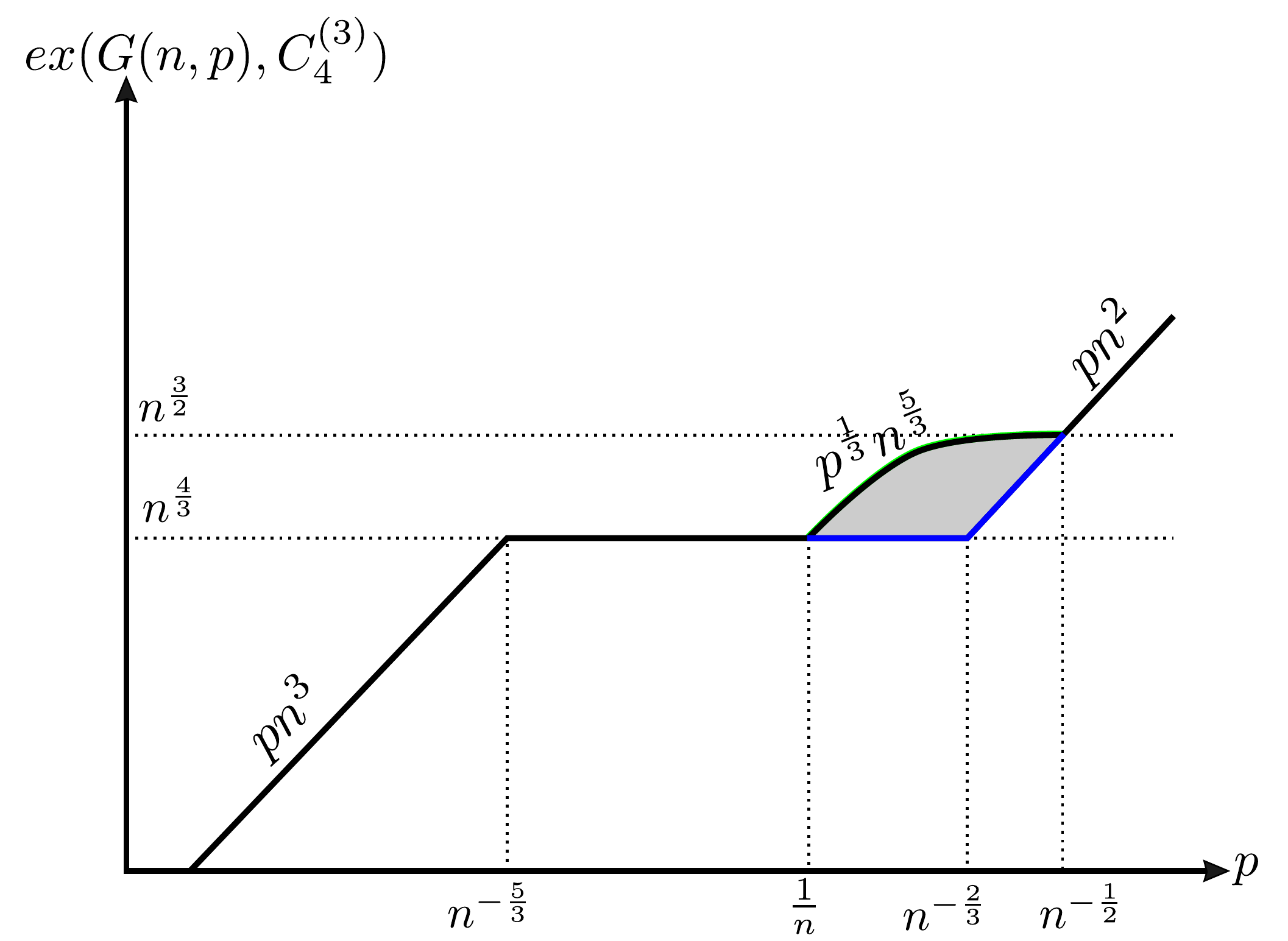}
  \caption{The behaviour of $\ex(G_{n,p}^{(3)}, C_4^{(3)})$}
  \label{c4}
\end{figure}
 Figure 2 shows that we do not know the optimal  behaviour (up to polylog factors) of  $\ex(G_{n,p}^{(3)}, C_{4}^{(3)})$ in the regime $1/n\leq p\leq 1/\sqrt{n}$. We know of various constructions which reach the lower bound $n^{4/3}$ up to polylog factors in the regime $n^{-1}\leq p\leq n^{-2/3}$. This leads  us to conjecture that $n^{4/3}$ is the correct growth rate for this range of $p$.

\begin{conjecture}$$ \ex\left(G_{n,p}^{(3)}, C_{4}^{(3)}\right) = \begin{cases} (1+o(1))e(G_{n,p}^{(3)}), & \mbox{if } {n^{-3}\ll p \ll n^{-5/3}},\\ 
\Theta(n^{4/3+o(1)}), & \mbox{if } {n^{-5/3}\ll p \ll n^{-2/3} },\\
\Theta(pn^{2}),& \mbox{otherwise.} \end{cases}$$
\end{conjecture}

 Here we describe two constructions which achieve the lower bound $n^{4/3}$ up to polylog factors as they might be of independent interest as well. Note that both of these constructions are not only  $C_4^{(3)}$-free but of girth at least five, where here girth is in the sense of \emph{Berge} cycle.  In comparison, the other construction that reached the lower bound $pn^2$ has very high co-degree. A \emph{Berge cycle} of length $k\geq 2$ in a hypergraph is an alternating sequence
of distinct vertices and edges $v_1, e_1, \dots , v_k, e_k$ such that $v_i
, v_{i+1} \in e_i$
for each $i$ (where indices are
taken modulo $k$). The \emph{girth} of an $r$-graph is the length of the shortest cycle.  

\medskip

\textbf{Blowups of Steiner systems:} For any $1< t\leq \sqrt{n/2}$,  let  $S$ be any partial $(n,t,2)$-Steiner system with $cn^2/t^2$ edges, for some positive constant $c$ (we can in fact take $c=1/4$). Recall that such a partial $(n,t,2)$-Steiner system  is a $t$-graph on $[n]$ where every pair of vertices is contained in at most one $t$-edge. We defer the proof of the existence of such partial Steiner systems to the end. Replace every $t$-edge in $S$ by a complete $3$-graph on $t$ vertices and intersect this graph with $G_{n,p}^{(3)}$ for $p=n^{-2/3}t^{-1} (\log{n})^{-1}$, denoting the resulting graph by $G$. The  expected number of edges in $G$ is    at least $ \Omega(p n^{2}t)$. Let $X$, $Y$, $Z$ denote the  number of  $2$, $3$ and $4$-cycles in $G$ correspondingly. It is not hard to see that   $\E[X]\leq  (c n^2/t^2) \cdot t^4 \cdot p^2  =  c(pnt)^2$, $\E(Y)\leq  cn^2/t^2 \cdot t^2\cdot n \cdot t^3 \cdot p^3 = c(pnt)^3$ and $E(Z)\leq  (cn^2/t^2)^2 \cdot t^8 p^4 =  c^2 (pnt)^4$. The choice of $p$ yields 
$\E[X], \E[Y], \E[Z] \ll \E[e(G)]$. Therefore for any $\beta >0$ by Markov's inequality, $\Prob[X \geq \beta \E[e(G)] ] \leq  \E[X]/\beta \E[e(G)] \rightarrow 0$. Thus a.a.s. $X= o( \E[e(G)] )$ and similarly,  $Y,Z= o( \E[e(G)] )$  as well. On the other hand,  $e(G)$ is a binomial variable and it is not hard to check that by the Chernoff bounds it is concentrated around its mean (using the value of $p$). Hence a.a.s. $e(G)\geq \Omega( p n^2t)$. So by deleting one edge from each $2$-cycle, $3$-cycle and $4$-cycle we obtain a subgraph of $G$ (and hence of $G_{n,p}^{(3)}$) of girth at least five with  $\Omega(pn^2t) = \Omega(n^{4/3}/\log{n})$  edges.   This construction is valid in the range $n^{-7/6}(\log{n})^{-1}\leq  p< (n^{-2/3})\log{n}^{-1}$.

To complete the proof, we need to show the existence of $S$. Given $n$ and $1<t<\sqrt{n/2}$, choose prime $q\ge t$ with $n/2t<q<n/t$ using Bertrand's postulate. We construct a partial $(qt, t,2)$-Steiner system on $qt$ vertices and $q^2$ edges and  add $n-qt$ isolated vertices. The resulting partial $(n,t,2)$-Steiner system has $n$ vertices and at least $n^2/4t^2$ edges.  The $(qt, t,2)$-Steiner system is constructed as follows. Let $V=\{(x,y): x \in [t]\subset \mathbb Z_q, y \in \mathbb Z_q\}$. For each $m,c \in \mathbb Z_q$, define the line $L(m,c) =\{(x,y) \in V: y=mx+c\}$.  For each $x$, there is a unique $y$ such that  $(x,y) \in L(m,c)$, hence $|L(m,c)|=t$. We let the $t$-sets in our Steiner system be the set of all lines. Every two lines have at most one point in common and the number of lines is $q^2$ as required. 
\medskip

\textbf{A construction based on high girth $3$-graphs:}  Let $n^{-2}\ll p\ll (\log{n})^{-2}$, choose $a=p^{-1/2}/6$,  and choose  $m=q^2$ such that $q$ is prime and $\sqrt{n/a}/2<q<\sqrt{n/a}$. Such $q$ exists by Bertrand's postulate. Let $H$ be a $3$-graph on $m$ vertices with $m^{3/2}/6$ edges and girth at least five. Such a $3$-graph exists by the results of Lazebnik and Verstra\"{e}te~\cite{LV}. Let $H(a)$ be the following $3$-graph obtained from $H$. We replace  each vertex $v$ of $H$ by a set $U_v$ of size $a$ and each edge $e$ of $H$ by a subgraph $K_{a}^{(3)}(e)$, a copy of $K_{a,a,a}^{(3)}$, the complete $3$-partite $3$-graph with partition classes of size $a$.  We also add isolated vertices so that the total number of vertices in $H(a)$ is $n$. We think of $H(a)$ and $G_{n,p}^{(3)}$ to be on the same (ordered) set of vertices.  Finally  let $G$ to be the following subgraph of $G_{n,p}^{(3)}$; We ignore all edges of $G_{n,p}^{(3)}$ which contain an isolated vertex of $H(a)$ or are completely contained in the sets $U_v$. For every edge $e\in E(H)$, we keep a maximal matching in the corresponding subgraph  $K_{a}^{(3)}(e)\cap G_{n,p}^{(3)}$. It is not hard to check that $H$ is $C_4^{(3)}$-free, in fact, it is of girth at least five. We claim that a.a.s $G$ has $\Omega(p^{1/4}n^{3/2})$ edges. For every edge $e\in E(H)$, let $M_{e,p}$ denote a maximum matching in $K_{a}^{(3)}(e)\cap G_{n,p}^{(3)}$. It is clear that $e(G)\geq \sum_{e\in E(H)}{|M_{e,p}|}$.  On the other hand, $e\in E(H)$, $|M_{e,p}|\geq X_{e,p}$, where $X_{e,p}$ is the number of isolated edges in $K_{a}^{(3)}(e)\cap G_{n,p}^{(3)}$. Note that 
$$\E[X_{e,p}]=a^3p(1-p)^{3(a^2-1)}\geq pa^3 (1-3a^2p)\geq pa^3/2.$$ Also, $X_{e,p}$ is a $3$-Lipshitz random variable. We set $\sigma=3a^{3/2}p^{1/2}(1-p)^{1/2}$ 
and $t=p^{1/2}a^{3/2}$. Apply Lemma~\ref{azuma} to obtain $$ \Prob[X_{e,p} < E[X_{e,p}]/2]\leq \Prob[|X_{e,p}-\E[X_{e,p}]|> t\sigma ] \leq 2e^{-t^2/4} = 2e^{-a/4}.$$
Taking a union bound over all edges $e\in E(H)$ we obtain that the probability that there is some $e\in E(H)$ for which $X_{e,p}\leq pa^{3/2}/4$ is at most $2m e^{-a/4}\rightarrow 0$, as $n \rightarrow \infty$ since $ a=p^{-1/2}/6\gg \log{n}$. So a.a.s  $X_{e,p}\leq pa^3/4$ simultaneously for all $e\in E(H)$, thus, a.a.s. 

$$e(G)\geq \sum_{e\in E(H)}{|M_{e,p}|} \geq  \sum_{e\in E(H)}{X_{e,p}}\geq \frac{m^{3/2}}{6} \cdot pa^{3}/4 = \Omega(p^{1/4}n^{3/2}).$$

This bound is at least $n^{4/3}$ for  $p\geq n^{-2/3}$, so at $p=n^{-2/3}$ we have a construction which achieves the lower bound $\Omega(n^{4/3})$. Clearly the bound $p^{1/4}n^{3/2}$ is much smaller than the bound $pn^2$  achieved by a star but the fact that it is of girth at least five makes it interesting. It is in fact related to   $\ex_{\textit{lin}}(G_{n,p}^{(3)}, C_4^{(3)})$ where we restrict our attention to linear subgraphs of $G_{n,p}^{(3)}$ only.  
\end{section}
\bigskip

{\bf Acknowledgments.} 
We wish to thank Jozsi Balogh, Wojtek Samotij and Rob Morris for several insightful discussions about this problem.

\end{document}